\documentclass[sts,a4paper,preprint,reqno,10pt]{imsart} 
\usepackage{amsthm,amsmath}
\RequirePackage[colorlinks,citecolor=blue,urlcolor=blue,breaklinks=true]{hyperref}

\usepackage{amssymb,a4wide}
\usepackage{amsfonts,mathrsfs}
\usepackage{url}
\usepackage{pifont,dsfont}
\usepackage{breakcites,microtype}
\usepackage[round,sort]{natbib}

\usepackage{cleveref}
\usepackage{autonum}

\usepackage{nicefrac}
\usepackage{graphicx}

\usepackage{cleveref}
\usepackage{color}

\usepackage[lined,longend,algoruled,linesnumbered,boxed]{algorithm2e} 
\DontPrintSemicolon
\usepackage{algorithmic}

\crefname{algocf}{alg.}{algs.}
\Crefname{algocf}{Algorithm}{Algorithms}

\newtheorem{lemma}{Lemma}

\newtheorem{theo}{Theorem}
\newtheorem{theorem}{Theorem}

\def\var{\mathbf{Var}}

\newcommand{\esp}{\mathbf{E}}

\def\fcar{\mathds{1}}
\newcommand{\KL}{{\rm KL}}

\newcommand{\bfX}{\mathbf{X}}

\newcommand{\bfM}{\mathbf{M}}
\newcommand{\bfE}{\mathbf{E}}
\newcommand{\bfP}{\mathbf{P}}
\newcommand{\bfQ}{\mathbf{Q}}

\newcommand{\bs}{\boldsymbol}
\newcommand{\bX}{\bs X}

\newcommand{\bZ}{\bs Z}
\newcommand{\bL}{\bs L}
\newcommand{\bF}{\bs F}

\newcommand{\bv}{\bs v}

\newcommand{\bk}{\bs k}

\newcommand{\bzeta}{\bs\zeta}

\newcommand{\bepsilon}{\bs\epsilon}
\newcommand{\bnu}{\bs\nu}
\newcommand{\bmu}{\bs\mu}

\newcommand{\bxi}{\bs\xi}
\newcommand{\bXi}{\bs\Xi}

\newcommand{\btheta}{\bs\theta}
\newcommand{\eeta}{\bs\eta}

\newcommand{\bfTheta}{\mathbf \Theta}

\newcommand{\RR}{\mathbb R}
\newcommand{\NN}{\mathbb N}

\def\calM{\mathcal M}
\def\calI{\mathcal I}

\def\LGHT{\hat\bL{}^{\rm GHT}}

\def\ie{\textit{i.e., }}

\def\hat{\widehat}

\definecolor{gris25}{gray}{0.90}

\begin{document}

\begin{frontmatter}

\title{Estimating linear functionals of a sparse family of Poisson means}
\runtitle{Estimating linear functionals of Poisson means}

\begin{aug}
\author{\fnms{Olivier} \snm{Collier,}\ead[label=e3]{}}
\author{\fnms{Arnak S.} \snm{Dalalyan}\ead[label=e2]{arnak.dalalyan@ensae.fr}}

\runauthor{O.\ Collier and A.\ Dalalyan}

\affiliation{Modal'X, Universit\'e Paris-Nanterre and CREST, ENSAE}

\address{Modal'X, UPL, Univ Paris Nanterre, F92000 Nanterre France.}
\address{3 avenue Pierre Larousse,
92245 Malakoff, France.}
\end{aug}

\begin{abstract}
Assume that we observe a sample of size $n$ composed of $p$-dimensional signals, each signal having independent
entries drawn from a scaled Poisson distribution with an unknown intensity. We are interested in estimating
the sum of the $n$ unknown intensity vectors, under the assumption that most of them coincide with a
given ``background'' signal. The number $s$ of $p$-dimensional signals different from the background 
signal plays the role of sparsity and the goal is to leverage this sparsity assumption in order to
improve the quality of estimation as compared to the naive estimator that computes the sum of the 
observed signals. We first introduce the group hard thresholding estimator and analyze its mean squared
error measured by the squared Euclidean norm. We establish a nonasymptotic upper bound showing that
the risk is at most of the order of $\sigma^2(sp+s^2\sqrt{p}\,\log^{3/2}(np))$. We then 
establish lower bounds
on the minimax risk over a properly defined class of collections of $s$-sparse signals. These lower bounds
match with the upper bound, up to logarithmic terms, when the dimension $p$ is fixed or of larger order than
$s^2$. In the case where the dimension $p$ increases but remains of smaller order than $s^2$, our results 
show a gap between the lower and the upper bounds, which can be up to order $\sqrt{p}$. 
\end{abstract}



\end{frontmatter}


\section{Introduction and problem formulation}

Let $\bmu=(\mu_1,\ldots,\mu_p)^\top$ be any vector from $\RR^p$ with positive entries. In what follows, we denote 
by $\mathcal P_p(\bmu)$ the distribution of a random vector $\bX$ with independent entries
$X_j$ drawn from a Poisson distribution $\mathcal P(\mu_j)$, for $j=1,\ldots,p$. We consider 
that the available data is composed of $n$ independent vectors $\bX_1,\ldots,\bX_n$ 
randomly drawn from $\RR^p$ such that
\begin{align}
\bX_{i} \stackrel{iid}{\sim}\mathcal \sigma^{2}\mathcal P_p(\sigma^{-2}\bmu_{i}),\qquad i=1,\ldots,n,
\end{align}
where $\sigma>0$ is a known parameter that can be interpreted as the noise level. The use of this term
is justified by the fact that when $\sigma$ goes to zero, we have 
$\sigma^{2}\mathcal P_p(\sigma^{-2}\bmu) \approx \mathcal N_p\big(\bmu, \sigma^2\text{diag}(\bmu)\big)$. 
We will work under the assumption that most intensities are known and equal to a given vector $\bmu_0$, 
which can be thought of as a background signal. We denote by 
$\bfX= [\bX_1,\ldots,\bX_n]$ the $p\times n$ matrix obtained by concatenating the vectors $\bX_i$ for $i=1,\ldots,n$.

\textbf{Assumption (S):} \textit{for a given vector $\bmu_0\in(0,\infty)^p$, the set 
$S=\{i:\bmu_i\not=\bmu_0\}$ is a very small subset of $[n]:=\{1,\ldots,n\}$.}

The unknown parameter in this problem is the $p\times n$ matrix $\bfM = [\bmu_1,\ldots,\bmu_n]$
obtained by concatenating the signal vectors $\bmu_i$. However, we will not aim at estimating the 
matrix $\bfM$. Instead, we aim at estimating a linear functional of the intensities: 
\begin{align}
\bL(\bfM) = \sum_{i\in[n]} (\bmu_i-\bmu_0) = \sum_{i\in S} (\bmu_i-\bmu_0) = \bfM \mathbf 1_n - n\bmu_0.
\end{align}
This functional has a clear meaning: it is the superposition of all the signals contained in $\bX_i$'s
beyond the background signal $\bmu_0$. Notice that if we knew the set $S$, it would be possible to use 
the oracle 
\begin{align}\label{oracle}
\hat\bL_S = \sum_{i\in S} (\bX_i-\bmu_0),
\end{align}
which would lead to the risk
\begin{align}
\bfE\big[\|\hat\bL_S-\bL(\bfM)\|_2^2\big] = \sum_{i\in S}\sum_{j\in[p]} \var(X_{ij}) = \sigma^2
\sum_{i\in S} \bmu_i^\top \mathbf 1_p\le \mu_\infty\sigma^2 sp ,
\end{align}
where $\mu_\infty$ is an upper bound on $\max_{i}\|\bmu_i\|_{\infty}$ and $s$ is the cardinality of $S$. On the other hand, one can use
the naive estimator 
$$
\hat\bL_{[n]} = \sum_{i\in [n]} (\bX_i-\bmu_0) = \bL(\bfX)
$$
which does not use at all the sparsity assumption (S). It has a quadratic risk given by
\begin{align}
\bfE\big[\|\hat\bL_{[n]}-\bL(\bfM)\|_2^2\big] = \sum_{i\in [n]}\sum_{j\in[p]} \var(X_{ij}) = \sigma^2
\sum_{i\in [n]} \bmu_i^\top \mathbf 1_p\le \mu_\infty\sigma^2 n p.
\end{align}

One can remark that the rate  $\sigma^2 sp$ we have got for the oracle is independent of $n$, which 
obviously can not continue to be true when $S$ is unknown and should somehow be inferred from the data.
An important statistical question  is thus the following. 

\textbf{Question (Q1):} \textit{What kind of quantitative improvement can we get upon the rate $\sigma^2 n p$ of 
the naive estimator by leveraging the sparsity assumption}? 

In the framework considered in the present work, for answering Question (Q1), we are allowed to use any 
estimator of $\bL(\bfM)$ which requires only the knowledge of $\sigma$ and $\bmu_0$. However, in practice,
it is common to give advantage to estimators having small computational complexity.  Therefore, another 
question to which we will give a partial answer in this work is:

\textbf{Question (Q2):} \textit{What kind of quantitative improvement can we get upon the rate $\sigma^2 n p$ of 
the naive estimator by leveraging the sparsity assumption and by constraining ourselves to estimators that are 
computable in polynomial time}? 

In the statement of the last question, the term ``polynomial time'' refers to the running time as a function of
$s$, $n$ and $p$.

\subsection{Relation to previous work}

Various aspects of the problem of estimation of a linear functional of an unknown high-dimensional or even 
infinite-dimensional parameter were studied in the literature, mostly focusing on the case of a functional
taking real values (as opposed to the vector valued functional considered in the present work). Early results 
for smooth functionals were obtained by \cite{Koshevnik}. Minimax estimation of linear functionals 
over various classes and models were thoroughly analyzed by \cite{donoho1987minimax, klemela2001,
Efromovich1994,Gol04,cai2004,cai2005, laurent2008,butucea2009,juditsky2009}. 
There is also a vast literature on studying the problem of estimating quadratic functionals 
\citep{DONOHO1990290,laurent2000,cai2006,BickelRitov}. 
Since the estimators of (quadratic) functionals can be often used as test statistics, the problem of 
estimating functionals has close relations with the problem of testing that were successfully exploited 
in \citep{comminges2012, comminges2013,CD11c,Lepski1999}. The problem of estimation of nonsmooth 
functionals was also tackled in the literature, see \citep{cai2011}. 

Most papers cited above deal with the Gaussian white noise model, Gaussian sequence model, or 
the model in which the observations are iid with a common (unknown) density function. Surprisingly, 
the results on estimation of functionals of the intensity of a Poisson process are very scarce. (We are only aware of the paper \citep{KUTOYANTS199843}, which in an asymptotic set-up established the asymptotic normality and efficiency of a natural plug-in estimator of a real valued smooth linear functional.) 
This is surprising, given the relevance of the Poisson distribution for modeling many real data sets arising, for instance, in image processing and biology. Furthermore, there are by now several 
comprehensive studies of nonparametric estimation and testing for Poisson processes, see 
\cite{Ingster2007,kutoyants1998,birge,reynaud-bouret2010} and the references therein.  
In addition, several statistical experiments related to Poisson processes were proved to be asymptotically 
equivalent to the Gaussian white noise or other important statistical 
experiments \citep{brown2004,Grama1998}.  However, the models related to the Poisson distribution demonstrate some notable differences as compared 
to the Gaussian models. From a statistical point of view, the most appealing differences are the heteroscedasticity of the Poisson 
distribution\footnote{By heteroscedasticity we understand here the fact that if the mean is unknown, then 
the variance is unknown as well and these two parameters are interrelated} and the fact that it has heavier 
tails than the Gaussian distribution. The impact of these differences on the statistical problems is amplified in high dimensional settings, as it can be seen in \citep{Jia,Ivanoff}. 

To complete this section, let us note that the investigation of the statistical problems related to functionals of
high-dimensional parameters under various types of sparsity constraints was recently carried out in several papers. 
The case of real valued linear and quadratic functionals was studied by \cite{collier2017} and \cite{ColComTV}, focusing on the Gaussian sequence model. \cite{Gassiat} analyzed the problem of the signal-to-noise ratio estimation in the linear regression model under various assumptions on the design. In a companion paper, \cite{CollierDal2} consider the problem of a vector valued linear functional estimation in the Gaussian sequence model. The relations with this work will be discussed in more details in \Cref{sec:conc} below.

\subsection{Agenda}

The rest of the paper is organized as follows. In \Cref{sec:GHT}, we introduce an estimator of the functional 
$\bL(\bfM)$ based on thresholding the columns of $\bfX$ with small Euclidean norm. The same section contains the
statement and the proof of a nonasymptotic upper bound on the expected risk of the aforementioned estimator. 
Lower bounds on the risk are presented in \Cref{sec:lower}. Some avenues for further research and a summary of
the contributions of the present work are given in \Cref{sec:conc}. The proofs of the technical lemmas are gathered
in \Cref{ssec:lemproofs}.

\subsection{Notation}
We use uppercase boldface letters for matrices and uppercase italic boldface letters for vectors. 
For any vector $\bv$ and for every $q\in(0,\infty)$, we use the standard definitions of the $\ell_q$-norms $\|\bv\|_q^q = \sum_j |v_j|^q$, along with $\|\bv\|_\infty = \max_j |v_j|$. For every $n\in\NN$, $\mathbf 1_n$ (resp.\ $\mathbf 0_n$) is the vector in $\RR^n$ all the entries of which are equal to $1$ (resp.\ $0$). For a matrix $\bfM$, we denote by 
$\|\bfM\|$ the spectral norm of  $\bfM$ (equal to its largest singular value),  and by $\|\bfM\|_F$ its Frobenius norm (equal to the Euclidean norm of the vector composed of its singular values). For any matrix $\bfM$, we denote by $\bL(\bfM)$ 
the linear functional defined as the sum of the columns of $\bfM$. 

\section{Group hard thresholding estimator}\label{sec:GHT}
This section describes the group hard thresholding estimator, that is computationally tractable and achieves the optimal convergence rate in a large variety of cases. The underlying idea is rather simple: the Euclidean distance between each signal $\bX_i$ and the background signal $\bmu_0$ is compared 
to a threshold; the sum of all the signals for which the 
distance exceeds the threshold is the retained estimator. 

Thus, the group hard thresholding method analysed in this section can be seen as a two-step estimator, that 
first estimates the set $S$ and then substitutes the estimator of $S$ in the oracle
$\hat\bL_S$ considered in~\eqref{oracle}. To perform the first step, we choose a threshold $\lambda >0$ 
and define 
\begin{align}\label{hatS}
\hat S(\lambda) = \Big\{i\in [n] : \|\bX_i-\bmu_0\|_2 \ge \sigma(\|\bX_i\|_1+\lambda)^{1/2}\Big\}.
\end{align}
We will refer to the plug-in estimator 
$\hat\bL_{\hat S(\lambda)}$ as the group hard thresholding 
estimator, and will denote it by
\begin{align}\label{GHT}
\LGHT = \sum_{i\in \hat S(\lambda)} (\bX_i-\bmu_0).
\end{align}
It is clear that the choice of the threshold $\lambda$ has a strong impact on the statistical accuracy of the
estimator $\LGHT$. One can already note that the threshold to 
which the distance $\|\bX_i-\bmu_0\|_2$ is compared depends on 
the signal $\bX_i$; this is due to the heteroscedasticity of
the models based on observations drawn from the  Poisson 
distribution.

Prior to describing in a quantitative way the impact of the threshold on the expected risk, let us emphasize that the computation of $\LGHT$ requires 
$O(np)$ operations; so it can be done in polynomial time. Indeed, the 
computation of each of the $n$ norms $\|\bX_i-\bmu_0\|_2$ requires $O(p)$ operations, so $\hat S(\lambda)$ can be
computed using $O(np)$ operations. The same is true for the sum in the right hand side of~\eqref{GHT}.

The following theorem establishes the performance of the previously defined thresholding estimator, for a
suitably chosen threshold $\lambda$.
\begin{theo}\label{theorem_HT}
	Define $\mu_\infty = \max_{i,j} |\bfM_{ij}|$ and suppose 
	$\sigma^3\mu_\infty^{-3/2} \le 40 \log^{3/2}({2np}) 
	\le 0.9 \sigma^{-3}\mu_\infty^{3/2} $. The group hard thresholding estimator \eqref{hatS}-\eqref{GHT} with
	\begin{equation}\label{lambda}
		\lambda = 40 \|\bmu_0\|_\infty \sqrt{p}\, \log^{3/2}({2np}),
	\end{equation}
	has a mean squared error bounded as follows:
	\begin{equation}
		\esp\big[\| \LGHT-\bL(\bfM) \|_2^2\big] \le \mu_\infty \sigma^2 \Big(
		170 s^2\sqrt{p}\,\log^{3/2}({2np}) + 6 sp\Big).
	\end{equation}
\end{theo}

\begin{proof}
Let us introduce the auxiliary notation $\bxi_i = \sigma^{-1}(\bX_i-\bmu_i)$ and write the 
decomposition
\begin{align}
	\LGHT-\bL(\bfM) 
		&= \sum_{i\in S} (\bX_i-\bmu_i) - \sum_{i\in S\setminus \hat S(\lambda)} (\bX_i-\bmu_0)+
			 \sum_{i\in \hat S(\lambda)\setminus S} (\bX_i-\bmu_0)\\
		&= \sum_{i\in S\cap \hat S(\lambda)} \sigma\bxi_i - \sum_{i\in S\setminus \hat S(\lambda)} (\bmu_i-\bmu_0)+
			 \sum_{i\in \hat S(\lambda)\setminus S} \sigma\bxi_i.\label{decomposition}
\end{align}
 We have
	\begin{align}\label{ineq:2}
	\big\| \LGHT-\bL(\bfM) \big\|_2 
		&\le \sigma\underbrace{\Big\| \sum_{i\in S\cap \hat S(\lambda)} \bxi_i \Big\|_2}_{T_1} + 
				 \underbrace{\Big\|\sum_{i\in S\setminus\hat S(\lambda)} (\bmu_i-\bmu_0)\Big\|_2}_{T_2} 
				 + \sigma \underbrace{\Big\|\sum_{i\in \hat S(\lambda)\setminus S} \bxi_i\Big\|_2}_{T_3}.
	\end{align}
Let us use the notation 
	$\bXi_S = [\bxi_i:i\in S]$ for the matrix	obtained by concatenating the vectors $\bxi_i$ with $i\in S$, 
	and denote by $\bxi_S^j$ its $j$-th row. We can write
	\begin{equation}
	T_1^2 = \|\bXi_S\mathbf 1_{\hat S(\lambda)}\|_2^2 \le  s\|\bXi_S\|^2
	    = s\bigg\|\sum_{j\in [p]} (\bxi_S^j)^\top\bxi_S^j\bigg\|,
	\end{equation}
	where $\mathbf 1_{\hat S(\lambda)}$ is the vector in $\{0,1\}^S$ having zeros at positions 
	$i\in \hat S(\lambda)$ and zeros elsewhere. One can observe that the $s\times s$ matrices 
	$(\bxi_S^j)^\top\bxi_S^j$ are independent and each of them is in expectation equal to the diagonal matrix
	$\text{diag}(\bmu_S^j)$. Since the operator norm of a diagonal matrix is equal to its largest (in absolute value) 
	entry, we have 
	\begin{align}
	\bfE\big[T_1^2\big]
			&\le \mu_\infty sp + s\bfE\bigg[\Big\|\sum_{j\in [p]} \big\{(\bxi_S^j)^\top\bxi_S^j-
					{\rm diag}(\bmu_S^j)\big\}\Big\|\bigg]\\
			&\le \mu_\infty sp+ s\bigg\{\bfE\bigg[\Big\|\sum_{j\in [p]} 
					\big\{(\bxi_S^j)^\top\bxi_S^j-{\rm diag}(\bmu_S^j)\big\}\Big\|^2\bigg]\bigg\}^{1/2}.						
	\end{align}		
Upper bounding the spectral norm by the Frobenius norm, we arrive at
	\begin{align}
	\bfE\bigg[\Big\|\sum_{j\in [p]} \big\{(\bxi_S^j)^\top\bxi_S^j-{\rm diag}(\bmu_S^j)\big\}\Big\|^2\bigg]
			&\le \bfE\bigg[\Big\|\sum_{j\in [p]} \big\{(\bxi_S^j)^\top\bxi_S^j-{\rm diag}(\bmu_S^j)\big\}\Big\|^2_F\bigg]\\
			&= \sum_{j\in [p]} \bfE\big[\big\|(\bxi_S^j)^\top\bxi_S^j-{\rm diag}(\bmu_S^j)\big\|^2_F\big].			
	\end{align}		
Using the fact that the fourth order central moment of the Poisson distribution 
$\mathcal P(a)$ is $a+3a^2$ and that $\xi_{ij} \sim \sigma(\mathcal P(\sigma^{-2}\mu_{0j})
-\sigma^{-2}\mu_{0j})$, one can check that 
\begin{align}\label{xiij}
\bfE[\xi_{ij}^4]= \sigma^2\mu_{ij}+3\mu_{ij}^2\le (2\mu_\infty)^2,
\end{align}
where the last inequality holds true since $\sigma^2\le (0.9)^{1/3}\mu_\infty$. 

In view of \eqref{xiij} and the Cauchy-Schwarz inequality, it holds that
$\bfE[\|(\bxi_S^j)^\top\bxi_S^j-\text{diag}(\bmu_S^j)\|^2_F] \le 2.5s\mu_\infty^2+(s^2-s)\mu_{\infty}^2
\le (s+1)^2\mu_\infty^2$. Putting these estimates together, we get
\begin{equation}
	\bfE\big[T_1^2\big] 
			\le \mu_\infty sp + \mu_\infty s(s+1)\sqrt{p}\le \mu_\infty sp+ 2\mu_\infty s^2\sqrt{p}.
			\label{eqn4}
\end{equation}	
We would like to stress here that this upper bound is quite rough, because of the step in which 
we upper bounded the spectral norm by the Frobenius one. Using more elaborated arguments on the expected
spectral norm of a random matrix, such as the Rudelson inequality \cite{Rudelson} and \citep[Cor. 13.9]{boucheron2013concentration} combined with
the symmetrisation argument \citep[Theorem 2.1]{Koltch}, it is possible to replace
the rate $s^2\sqrt{p}$ by $s(\sqrt{s}+\sqrt{p})$. However, we do not give more details on this computations 
since the rate $s^2\sqrt{p}$ appears anyway as an upper bound on $\bfE[T_2]$.

To upper bound the term $T_2$ in \eqref{ineq:2}, we define $ \tau_i = \sigma^2(\|\bX_i\|_1+\lambda)$ and
$u=40\log^{3/2}(2np)$, so that $\tau_i = \sigma^2(\|\bX_i\|_1+\|\bmu_0\|_\infty\sqrt{p}\,u)$.  Note that $i\in S\setminus\hat S(\lambda)$ implies
	\begin{align}
	\|\bmu_i-\bmu_0\|_2^2
			&<\tau_i-\sigma^2\|\bxi_i\|^2	-2\sigma(\bmu_i-\bmu_0)^\top\bxi_i \phantom{\frac{|(\bmu_i-\bmu_0)^\top\bxi_i|^2}{\|\bmu_i-\bmu_0\|_2^2}}\\
			&<\tau_i-\sigma^2\|\bxi_i\|^2	+2\sigma|(\bmu_i-\bmu_0)^\top\bxi_i| \phantom{\frac{|(\bmu_i-\bmu_0)^\top\bxi_i|^2}{\|\bmu_i-\bmu_0\|_2^2}}\\
			&<\tau_i-\sigma^2\|\bxi_i\|^2	+0.5\|\bmu_i-\bmu_0\|_2^2+ 
			2\sigma^2\frac{|(\bmu_i-\bmu_0)^\top\bxi_i|^2}{\|\bmu_i-\bmu_0\|_2^2}.
	\end{align}
	In other terms, setting $\btheta_i = (\bmu_i-\bmu_0)/\|\bmu_i-\bmu_0\|_2$, we have 
	\begin{align}
	S\setminus\hat S(\lambda) \subset \bigg\{\|\bmu_i-\bmu_0\|_2^2
			&<2\tau_i-2\sigma^2\|\bxi_i\|^2	+
			4\sigma^2|\btheta_i^\top\bxi_i|^2\bigg\}.
	\end{align}
	As a consequence,
	\begin{align}
	T_2^2 &\le s \sum_{i\in S\setminus\hat S(\lambda) } \|\bmu_i-\bmu_0\|_2^2 \\
			&\le 2\mu_\infty s^2\sigma^2\sqrt{p}\,u+2\sigma^2s\sum_{i\in S}  (\|\bX_i\|_1-\|\bxi_i\|_2^2)_+
			+4\sigma^2s \sum_{i\in S} (\btheta_i^\top\bxi_i)^2.
	\end{align}
 Taking the expectation of both sides, in view of the independence of the components of $\bxi_i$ and those of $\bX_i$
and the condition $\sigma^2\le \mu_\infty$, we arrive at
	\begin{align}
	\bfE\big[T_2^2\big]  
	&\le 2\mu_\infty \sigma^2s^2\sqrt{p}\,u+2\mu_\infty\sigma^2s^2 \sqrt{6p}\,
			+4\mu_\infty\sigma^2s^2\\
	&\le 89\mu_\infty \sigma^2s^2\sqrt{p}\,\log^{3/2}({2np}).\label{eqn3}
	\end{align}		
The third term in \eqref{ineq:2} corresponds to the second kind error in the problem of support estimation and its expectation can be upper bounded using the Cauchy-Schwarz inequality as follows
\begin{align}
	\bfE[T_3^2] 
		&\le n\sum_{i\not\in S} \bfE[\|\bxi_i\|_2^2\fcar_{\|\bxi_i\|^2_2\ge \tau_i/\sigma^2}]\\
		&= n\sum_{i\not\in S}\sum_{j\in[p]} \bfE[\xi_{ij}^2\fcar_{\|\bxi_i\|^2_2\ge \tau_i/\sigma^2}]\\
		&\le n\sum_{i\not\in S}\sum_{j\in[p]} \big\{\bfE[\xi_{ij}^4]\bfP(\|\bxi_i\|^2_2\ge \tau_i/\sigma^2)\big\}^{1/2}.
\end{align}
In view of \eqref{xiij}, this yields
\begin{align}\label{T3b}
	\bfE[T_3^2] 
		&\le 2\mu_\infty np\sum_{i\not\in S}\bfP(\|\bxi_i\|^2_2\ge \tau_i/\sigma^2)^{1/2}.
\end{align}
For the last probability, we have
\begin{align}
\bfP(\|\bxi_i\|^2_2\ge \tau_i/\sigma^2)
		& = \bfP(\|\bxi_i\|^2_2 - \|\bX_i\|_1 \ge \|\bmu_0\|_\infty \sqrt{p}\, u).
\end{align}
To upper bound this probability, we apply \Cref{lem:1} below\footnote{The proof of \Cref{lem:1} 
is postponed to \Cref{ssec:lemproofs}} to 
$\eeta = \sigma^{-2}\bX_j$ and $\bnu= \sigma^{-2}\bmu_j$.
\begin{lemma}\label{lem:1}
Let $\eta_1,\ldots,\eta_p$ be independent random variables drawn from the Poisson 
distributions $\mathcal P(\nu_1),\ldots,\mathcal P(\nu_p)$, respectively. We set 
$\eeta= (\eta_1,\ldots,\eta_p)$ and $\bnu = (\nu_1,\ldots,\nu_p)$. Then, 
for every $u$ satisfying $\nu_\infty^{-3/2}\le u\le 0.9\nu_{\infty}^{3/2}$, we have 
\begin{align}
\bfP\Big(\|\eeta-\bnu\|_2^2-\|\eeta\|_1\ge \nu_\infty\sqrt{p}\,u\Big) 
		\le  (2p+1)e^{-cu^{2/3}}.
\end{align}
where $c\ge 0.38$ is a universal constant.
\end{lemma}

This yields that for $\sigma^3\mu_\infty^{-3/2}\le u\le 0.9 \sigma^{-3}\mu_\infty^{3/2}$, 
we have
\begin{align}
\bfP(\|\bxi_i\|^2_2\ge \tau_i/\sigma^2)
		&\le (2p+1)e^{-cu^{2/3}},
\end{align}
where $c\ge 0.38$ is an absolute constant. In view of \eqref{T3b}, this implies the 
following upper bound  on the expectation of $T_3$:
\begin{align}
	\bfE[T_3^2] 
		&\le 2 \mu_\infty n^2p(2p+1)^{1/2}e^{-0.5cu^{2/3}}\\
		&\le 4\mu_\infty n^2p^2 e^{-0.5cu^{2/3}}.
\end{align}
The choice  $u = 40 \log^{3/2}(2np)$ 
ensures that
\begin{align}
	\bfE[T_3^2] &\le 4\mu_\infty n^2p^2 e^{-2\log({2np})}\le \mu_\infty,\label{T3c}
\end{align}
which is negligible with respect to the upper bound on $\bfE[T_1^2]$ obtained above.

The result follows from upper bounds \eqref{ineq:2}, \eqref{eqn4}, \eqref{eqn3} and \eqref{T3c}, 
as well as the choice of $u$. Indeed, by the Minkowskii inequality, we get
\begin{align}
\Big(\bfE[\|\LGHT-\bL(\bfM)\|_2^2]\Big)^{1/2} \le 11.85 \sigma\mu_\infty^{1/2}sp^{1/4}\log^{3/4}({2np}) + \sigma(\mu_\infty sp)^{1/2}.
\end{align}
Taking the square of both sides and using the inequality $(a+b)^2\le (\nicefrac65)a^2+6b^2$, we get the claim 
of the theorem. 
\end{proof}

\section{Lower bounds on the minimax risk over the sparsity set}\label{sec:lower}

Let us denote by $\calM(n,p,s,\bmu_0,\mu_\infty)$ the set of all $p\times n$ matrices 
with positive entries bounded by $\mu_\infty$ and having at most $s$ columns different 
from $\bmu_0$. This implicitly means, of course, that $\mu_\infty\ge \|\bmu_0\|_\infty$, otherwise the set is empty. The parameter
set $\calM(n,p,s,\bmu_0,\mu_\infty)$ is the sparsity class over which the minimax risk 
will be considered. The result of the previous section implies that the minimax risk
over $\calM(n,p,s,\bmu_0,\mu_\infty)$ is at most of order $\mu_\infty\sigma^2s(p+s\sqrt{p})$, 
up to a logarithmic factor. In the present section, we will establish two lower bounds 
on the minimax risk over $\calM(n,p,s,\bmu_0,\mu_\infty)$.

The first lower bound shows that the minimax risk is at least of order $s^2$, up to log factor. This implies that the group hard thresholding estimator defined in previous section is rate-optimal when $p$ is fixed and $s$ goes to infinity. The proof of this lower bound relies on the comparison of the Kullback-Leibler divergence between the probability distribution of $\bfX$ corresponding to $\bfM_0 = [\bmu_0,\ldots,\bmu_0]$, and the 
one corresponding to a matrix $\bfM$, the $n$ columns of which are randomly drawn from the set $\{\bmu_0,\bmu_1\}$ with proportion $(n-s):s$. When $\bmu_1$ is well chosen, the two
aforementioned distributions are close, whereas the corresponding linear functional $\bL(\bfM_0)$ is sufficiently away from $\bL(\bfM)$. 

\begin{theorem}\label{theorem_lowerbound1}
	Assume that $s\ge 128$, then 
	\begin{equation}
	\inf_{\hat\bL} \sup_{\bfM\in\calM(n,p,s,\mu_0,\mu_\infty)} 
	\esp_\bfM \big\| \hat\bL - \bL(\bfM) \big\|^2 
	\ge \frac{\|\bmu_0\|_\infty}{513}\sigma^2 s^2\log\Big(1+\frac{n}{2s^2}\Big),
	\end{equation}
	where the infimum is taken over all estimators of the linear functional.
\end{theorem}

\begin{proof}
	We will make use of some lemmas, the statement of which requires additional notation. 
	As usual in lower bounds, we will make use of the Bayes risk: for every measure $\pi$ 
	on the set of $p\times n$ real matrices, we define the probability measure  
	\begin{align}
		\bfP_\pi(\,\cdot\,) = \int_{\RR^{p\times n}} \bfP_\bfM(\,\cdot\,)\,\pi(d\bfM).
		\end{align}
	Let $\bF:\RR_+^{p\times n}\to\RR^m$ be any functional, not necessarily linear, defined on the set of matrices.
	
	\begin{lemma}\label{lemma_lowerbound}
		Assume that $\pi_0$ and $\pi_1$ are two probability measures on the set of $p\times n$ matrices, $\RR_+^{p\times n}$, and set 
		$M = \max_i \var_{\pi_i}(\bF)$. Furthermore, suppose 	that $\KL(\bfP_{\pi_1},\bfP_{\pi_0}) \leq \nicefrac18$ and, for some $v>0$,
		\begin{equation}
		\big\|\bfE_{\pi_0}(\bF)-\bfE_{\pi_1}(\bF)\big\|_2>
		2\big(v+2\sqrt{M}\big). 
		\end{equation}
		Then, for any estimator $\hat\bF$ of $\bF(\bfM)$, we have
		\begin{equation}
		\sup_{\bfM\in\mathcal M} \esp_\bfM\big\|\hat{\bF}-\bF(\bfM)\big\|_2^2 \ge 
		\big( \nicefrac18 - 
		\max_i\pi_i(\mathcal{M}^{\complement})\big)\,v^2.
		\end{equation}		
	\end{lemma}

	\begin{lemma}\label{KL1}
		Let $\mu_0>0$ and $\epsilon>0$ be two real numbers.	Define $\pi_0 = \delta_{\mu_0}^{\otimes n}$ as the Dirac 
		mass concentrated in 
		$(\mu_0,\ldots,\mu_0)\in\RR_+^{n}$ and $\pi_1= \lambda^{\otimes n}$ as the product 
		measure of the mixture
		$$
		\lambda = \Big(1-\frac{s}{2n}\Big) \delta_{\mu_0} + \Big(\frac{s}{2n}\Big)\delta_{\mu_0+\sigma^2\varepsilon}.
		$$ 
		Then
		\begin{equation}
		\KL(\bfP_{\pi_1}\|\bfP_{\pi_0}) \le \frac{s^2}{4n}\big(e^{\sigma^2\varepsilon^2/\mu_0}-1\big).
		\end{equation}
	\end{lemma}

	We are now in a position to establish the lower bound of \Cref{theorem_lowerbound1}. 
	For simplicity and without loss of generality, let us assume that the first 
	entry of $\bmu_0$, $\mu_{01}$, is the largest one. We choose two probability 
	measures $\pi^p_0$ and $\pi^p_1$ on the set of $p\times n$ (intensity) matrices 
	with positive entries, $\RR_+^{p\times n}$, in the following way. The measure 
	$\pi^p_0$ is simply the Dirac mass at $\bfM_0=(\bmu_0,\ldots,\bmu_0)$. As for 
	$\pi^p_1$, its marginal distribution of the first row equals the measure	
	$\pi_1$ defined in \Cref{KL1},  whereas all the other rows are deterministically 
	equal to the corresponding rows of the matrix $\bfM_0$. So, for two matrices 
	$\bfM$ and $\bfM'$ randomly drawn from $\pi_0^p$ and $\pi_1^p$, all the rows 
	except the first are equal (and equal to the corresponding row of $\bfM_0$).
	This implies that choosing the real $\varepsilon$ by the formula 
	\begin{align}
	    \varepsilon^2 = {(\mu_{01}/{\sigma^2})\log(1+\nicefrac{n}{2s^2})},
	\end{align}
	we get (in view of \Cref{KL1}) $\KL(\bfP_{\pi^p_1}\|\bfP_{\pi_0^p})\le 1/8$.

	Furthermore, we have
	\begin{equation}
	\esp_{\pi_0}[\bL(\bfM)]=\mathbf 0_p, \quad \esp_{\pi_1}[\bL(\bfM)]=
	\frac{1}{2}\,s\sigma^2\begin{bmatrix} \varepsilon\\ \mathbf 0_{p-1} \end{bmatrix},
	\end{equation}
	and
	\begin{equation}
	\var_{\pi_0}[\bL(\bfM)]=0, \quad \var_{\pi_1}[\bL(\bfM)] = 
	\frac12\Big(1-\frac{s}{2n}\Big)s\sigma^4 \varepsilon^2.
	\end{equation}
	This means that, after upper bounding $s$ by $s^2/128$, the constant $M$ of \Cref{lemma_lowerbound} 
	can be chosen as $M= 2^{-8}s^2\sigma^4\varepsilon^2$. Hence, for $v=(\nicefrac18)\sigma^2s\varepsilon$, we have 
	\begin{align}\label{cond1}
	    \big\|\esp_{\pi_0}[\bL(\bfM)]-\esp_{\pi_1}[\bL(\bfM)]\big\|_2
	    &= \frac{s\sigma^2\varepsilon}{2} = 2v + 4\sqrt{M}.
	\end{align}
	This implies that the two conditions of \Cref{lemma_lowerbound} are fulfilled
	with
	$$
	v = \frac{\sigma^2s\varepsilon}{8} = \frac{\sigma s\sqrt{\mu_{01}}}{8}\,\log^{1/2}(1+\nicefrac{n}{2s^2}).
	$$	
	Therefore, taking into account the fact that $\pi^p_0(\calM(n,p,s,\bmu_0,\mu_\infty))=1$, we have
	\begin{align}
	\sup_{\bfM\in\mathcal M(n,p,s,\bmu_0,\mu_\infty)} \esp_\bfM\big\|\hat{\bL}-\bL(\bfM)\big\|_2^2 \ge 
		\big( \nicefrac18 - 
		\pi_1^p(\calM(n,p,s,\bmu_0,\mu_\infty)^{\complement})\big)\,v^2.
	\end{align}
	To evaluate the probability $\pi_1^p(\calM(n,p,s,\bmu_0,\mu_\infty)^{\complement})$, 
	we remark that 
	it is equal to $\bfP(\zeta >s)$, where $\zeta$ is a random variable drawn from the binomial distribution
	$\mathcal B(n,(\nicefrac{s}{2n}))$. Using the Bernstein inequality \citep[p.\ 440]{ShorackWellner86book},  
	one can check that $\bfP(\zeta >s)\le e^{-3s/5}\le e^{-76}$, provided that $s\ge 128$. This completes 
	the proof.
\end{proof}

It follows from \Cref{theorem_lowerbound1} that the 
factor $s^2$ present in the upper bound of the 
minimax risk is unavoidable. The next result establishes another
lower bound showing the the term $sp$ present in the upper bound is unavoidable as well. 

\begin{theorem}\label{theorem_lowerbound2}
	For all $1\le s\le n$, if $p\ge16$, then
	\begin{equation}
		\inf_{\hat\bL} \sup_{\bfM\in\calM(n,p,s,\bmu_0,\mu_\infty)} \esp_\bfM \big\| \hat\bL - \bL(\bfM) \big\|^2 
		\ge 2^{-14}{\mu_\infty\sigma^2sp}.
	\end{equation}
\end{theorem}

\begin{proof}
	We set $\varepsilon^2=2^{-7}(\nicefrac{\sigma^2}{s\mu_\infty})$ and, for any $T\subset[p]$, define the matrix 
	$\bfM_T$ by
	\begin{equation}
		\big(\bfM_T\big)_{i,j} = 
		\begin{cases}
			\mu_\infty(1-\varepsilon), &\quad \text{ if } (i,j)\in T\times [s] \\
			\mu_\infty, &\quad \text{ if } (i,j)\in T^\complement\times [s], \\
			\mu_{0,i}, & \quad j\not\in [s].
		\end{cases}
	\end{equation}
	We also set $\bfM_0=\bfM_{\varnothing}$. It is clear that all these matrices $\bfM_T$ 
	belong to the sparsity class $\calM(n,p,s,\bmu_0,\mu_\infty)$. In addition, for any $T$, 
	we have
	\begin{equation}
		\KL\big(\bfP_{\bfM_T}\|\bfP_{\bfM_0}\big) 
			= \frac{s\mu_\infty}{{\sigma^2}} \sum_{j\in T} \Big[ {(1-\varepsilon)}\log(1-\varepsilon) + \varepsilon \Big] 
			\le \frac{sp\varepsilon^2\mu_\infty}{\sigma^2}.
	\end{equation}
	According to the Varshamov-Gilbert bound \citep[Lemma~2.9]{Tsybakov2009}, there exist 
	$ m= [e^{p/8}]+1$ sets $T_0,\ldots,T_m\subset[p]$ such that $T_0=\varnothing$ 
	and\footnote{In what follows, $\bigtriangleup$ stands for the symmetric difference of two sets.}
	\begin{equation}
		i\neq j \quad \Longrightarrow \quad |T_i\bigtriangleup T_j| \ge \frac{p}8.
	\end{equation}
	With this choice of sets, we have for all $i\neq j$,
	\begin{equation}
		\big\|\bL\big(\bfM_{T_i}\big)-\bL\big(\bfM_{T_j}\big)\big\|^2 
			\ge {s^2|T_i \bigtriangleup T_j|\varepsilon^2\mu_\infty^2} 
			\ge \frac{s^2p\varepsilon^2\mu_\infty^2}{2^3}.
	\end{equation}
	Finally, using Theorem~2.5 in~\citep{Tsybakov2009} with $\alpha = 1/16$, we get
	\begin{align}
		\inf_{\hat\bL} \sup_{\bfM\in\calM(n,p,s,\mu_0,\mu_\infty)} \esp_\bfM \big\| \hat\bL - \bL(\bfM) \big\|^2 
		&\ge (1-e^{-p/16})\times\Big(1-\frac1{8}-\frac{1}{\sqrt{p}}\Big) 
				\frac{s^2p\varepsilon^2\mu_\infty^2}{2^5}\\
		&\ge 2^{-14}{\sigma^2sp\mu_\infty}.
	\end{align}
	and the result follows.
\end{proof}

\begin{figure}
\includegraphics[width = 0.45\textwidth]{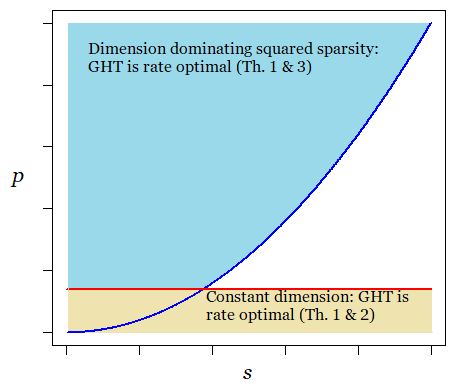}
\caption{The rate optimality diagram in the minimax sense for the group hard thresholding estimator.}
\label{fig:1}
\end{figure}

\section{Concluding remarks and outlook}\label{sec:conc}

Combining the results of the previous sections one can draw the optimality diagram, see
Fig.~\ref{fig:1}. In particular, \Cref{theorem_HT} and  
\Cref{theorem_lowerbound1} imply that the group hard thresholding estimator  \eqref{hatS}, 
\eqref{GHT} with the adaptively chosen threshold  \eqref{lambda} is rate optimal when the 
term $sp$ dominates $s^2\sqrt{p}\,\log^{3/2}(2np)$. This corresponds to the ``dimension 
dominates squared sparsity'' regime $p\ge s^2\log^3(2np)$. According to \Cref{theorem_HT} 
and \Cref{theorem_lowerbound2},  the estimator defined by  \eqref{hatS}, \eqref{GHT}, 
\eqref{lambda} is rate optimal in the low dimensional setting  $p = O(1)$ as well. The 
range of values of  $p$ in which the determination of the minimax rate remains open 
corresponds to the regime $p\to\infty$ so that $p/s^2 = O(1)$. 

In a companion paper \citep{CollierDal2}, the problem of estimating the linear functional 
$\bL(\bfM)$ is considered, when the observations are Gaussian vectors with unknown means 
$\bmu_i$. It turns out that in the Gaussian setting the optimal scaling of the minimax 
risk is $ps+ s^2$, up to a logarithmic factor. However, this rate is achieved by an estimator 
that has a computational complexity that scales exponentially with $n$. Furthermore, in the 
Gaussian model, the upper bound relies heavily on the fact that the Gaussian distribution is 
sub-Gaussian. The fact that the Poisson distribution is not sub-Gaussian (even though it is 
sub-exponential), does  not allow us to apply the methodology developed in the Gaussian model. 

It is worth mentioning here that although the risk bounds (both lower and upper) depend
on the sparsity level $s$ and the largest intensity $\mu_\infty$, the group hard thresholding
estimator is completely adaptive with respect to these parameters. 

The problem we have considered in this work assumes that the noise variance $\sigma^2$ and the
background signal $\bmu_0$ are known. It is not very difficult to check that the result of \Cref{theorem_HT} can be extended to the case where $\bmu_0$ is unknown, but an additional 
data set $\bar\bX_1,\ldots,\bar\bX_m$ is available, such that $\bar\bX_i$'s are iid vectors drawn from the distribution $\sigma^2\mathcal P_p(\bmu_0/\sigma^2)$. One can then estimate the
background signal $\bmu_0$ by the empirical mean of $\bar\bX_i$'s. As for $\sigma$, one can easily estimate this parameter using the fact that the entries of the data matrix $\bfX$ take their 
values in the set $\sigma^2\NN$.
Next, these estimates can be plugged in the group hard thresholding estimator. If $m$ is large enough, the error of estimating $\bmu_0$ and $\sigma^2$ is negligible with respect to the error of estimating $\bL(\bfM)$. Determining more precise risk bounds clarifying the impact of $m$ on the risk is an interesting avenue for future research.

\section{Proofs of the lemmas}\label{ssec:lemproofs}

\begin{proof}[Proof of \Cref{lem:1}]
Let us introduce an auxiliary real number $\gamma >1.5^2$ and denote 
\begin{align}
Z_j = (\eta_j-\nu_j-0.5)^2,\qquad Z_{j,\gamma} = Z_j\wedge \gamma,\qquad
\bar Z_{j,\gamma} = (Z_j-\gamma)_+
\end{align}
as well as
\begin{align}
\nu_{j,\gamma} = \bfE[Z_{j,\gamma}],\qquad \bar\nu_{j,\gamma} = \bfE[\bar Z_{j,\gamma}].
\end{align}
It is clear that $\nu_{j,\gamma} + \bar \nu_{j,\gamma} = \nu_j+0.25$ and we have
\begin{align}
\|\eeta-\bnu\|_2^2-\|\eeta\|_1 
		&= \sum_{j\in[p]} (Z_{j,\gamma}-\nu_{j,\gamma}) + \sum_{j\in[p]} (\bar Z_{j,\gamma}-\bar\nu_{j,\gamma})\\
		& = \|\bZ_{\gamma}\|_1 - \|\bnu_{\gamma}\|_1 + \|\bar\bZ_{\gamma}\|_1 - \|\bar\bnu_{\gamma}\|_1.\label{Z1'}
\end{align}
Since the random variables $Z_{j,\gamma}$ are independent and bounded, with 
$\bfP(Z_{j,\gamma}\in[0,\gamma])=1$, the Hoeffding inequality yields
\begin{align}\label{Z2'}
\bfP\Big(\|\bZ_{\gamma}\|_1-\|\bnu_{\gamma}\|_1 > t\Big) \le e^{-2t^2/(p\gamma^2)}.
\end{align}
On the other hand, it is clear that
\begin{align}
\bfP\Big(\sum_{j\in[p]}\nolimits (\bar Z_{j,\gamma}-\bar\nu_{j,\gamma}) > 0\Big) 
	&\le \bfP\big(\max_{j\in[p]}\nolimits Z_j> \gamma\big)\\
	&\le p\/\max_{j\in[p]}\bfP\big(|\eta_j-\nu_j-0.5|> \sqrt{\gamma}\big)\\
	&\le p\/\max_{j\in[p]}\bfP\big(|\eta_j-\nu_j|> \sqrt{\gamma}-0.5\big)\\
	&\le p\/\max_{j\in[p]}\bfP\big(|\eta_j-\nu_j|> (2/3)\sqrt{\gamma}\big).
\end{align}
Using the Bennet inequality for the Poisson random variables---see, for instance, \cite[Problem 15]{Pollard}--- 
we get\footnote{We use here the fact that the Bennet $\psi^*$ satisfies $\psi^*(x)\ge 0.772$ when $x\in[0,1]$.}
\begin{align}
\bfP\big(|\eta_j-\nu_j|> (2/3)\sqrt{\gamma}\big)
	&\le \bfP\big(\eta_j> \nu_j+ (2/3)\sqrt{\gamma}\big) + \bfP\big(\eta_j<\nu_j-(2/3)\sqrt{\gamma}\big)\\
	&\le e^{-1.54\gamma/(9\nu_j)}+e^{-2\gamma/(9\nu_j)}\le 2e^{-\gamma/(6\nu_j)},\qquad \forall \gamma\le (9/4)\nu_\infty^2.\label{Z3'}
\end{align}
Combining \eqref{Z1'}, \eqref{Z2'} and \eqref{Z3'}, we get
\begin{align}
\bfP\Big(\|\eeta-\bnu\|_2^2-\|\eeta\|_1\ge t\Big) 
		\le  \inf_{\gamma\in[(3/2)^2,(9/4)\nu_\infty^2]} \Big\{e^{-2t^2/(p\gamma^2)} + 2pe^{-\gamma/(6\nu_\infty)}\Big\}.
\end{align}
Let us set $t=\nu_\infty\sqrt{p}\,u$ and $\gamma^3=12\, u^2 \nu_\infty^3$. For $u$ such that 
$\nu_\infty^{-3/2}\le u\le 0.9\nu_\infty^{3/2}$,  the aforementioned value of $\gamma$ belongs 
to the range over which the infimum of the last display is computed. This yields
\begin{align}
\bfP\Big( \|\eeta-\bnu\|_2^2-\|\eeta\|_1\ge \nu_\infty\sqrt{p}\,u\Big) 
		\le  (2p+1)e^{-(12)^{1/3}u^{2/3}/6},
\end{align}
and the claim of the lemma follows.
\end{proof}

	\begin{proof}[Proof of \Cref{lemma_lowerbound}]
		Using Markov's inequality, we get
		\begin{align}
		\sup_{\bfM\in\mathcal{M}} v^{-2}\bfE_\bfM\big\|\hat{\bF}-\bF(\bfM)\big\|_2^2 
				&\ge \sup_{\bfM\in\mathcal{M}} \bfP_\bfM\big(\|\hat{\bF}-\bF(\bfM)\|_2\ge v\big) \\
				&\ge \max_{i=0,1} \bfP_{\pi_i}\big(\|\hat{\bF}-\bF(\bfM)\|_2\ge v\big) 
				- \max_{i=0,1} \pi_i(\mathcal{M}^{\complement}).
		\end{align}
		Now, using the test statistic $\psi=\arg\min_{i\in\{0,1\}} \big\|\hat{\bF}-\esp_{\pi_i}(\bF)\big\|_2$ and the notation
		$M=\max_i \var_{\pi_i}(F) = \max_i \bfE_{\pi_i}\|\bF-\bfE_{\pi_i}[\bF]\|_2^2$, we have
		\begin{align}
		\bfP_{\pi_i} (\psi\neq i) 
				&= \bfP_{\pi_i} \Big(\big\|\hat{\bF}-\esp_{\pi_{1-i}}(\bF)\big\|_2\le \big\|\hat{\bF}-\esp_{\pi_i}(\bF)\big\|_2\Big) \\
				&\le \bfP_{\pi_i} \Big(\big\|\hat{\bF}-\esp_{\pi_i}(\bF)\big\|_2 \ge v+\sqrt{M/\alpha}\Big).
		\end{align}
		Finally, using the triangle inequality again,
		\begin{align}
		\bfP_{\pi_i} (\psi\neq i) &\le \bfP_{\pi_i} \Big( \|\hat{\bF}-\bF(\bfM)\|_2 \ge v\Big) + \bfP_{\pi_i} \Big( \big\|\bF(\bfM)-\esp_{\pi_i}(\bF)\big\|_2 > \sqrt{\alpha^{-1} M} \Big) \\
		&\le  \alpha + \bfP_{\pi_i} \Big( \|\hat{\bF}-\bF(\bfM)\|_2 \ge v \Big),
		\end{align}
		so that
		\begin{equation}
		\sup_{\bfM\in\mathcal{M}} \bfP_\bfM\big(\|\hat{\bF}-\bF(\bfM)\|_2\ge v\big) \ge 
		\max_{i=0,1} \bfP_{\pi_i}(\psi\neq i) -\alpha - \max_{i=0,1} \pi_i(\mathcal{M}^{\complement}).
		\end{equation}
		The result is now a direct application of Theorem~2.2 in \citet*{Tsybakov2009} using the assumption 
		on the Kullback-Leibler divergence.
	\end{proof}

	\begin{proof}[Proof of \Cref{KL1}]
		Let us denote by  $\bfQ_{\mu}$, for every vector $\mu\in(0,\infty)$, the Poisson distribution 
		$\mathcal P(\sigma^{-2}\mu)$, and by $\bfQ_\lambda = \int_{(0,\infty)} \bfQ_{\mu}\,\lambda(d\mu)$, 
		the $\lambda$-mixture of Poisson distributions. We have, for every $X\in\mathbb N$,
		\begin{align}
		\frac{d\,\bfQ_{\lambda}}{d\,\bfQ_{\mu_0}}(X) 
			&= 1 + \frac{s}{2n} \Big(\frac{d\,\bfQ_{\mu_0+\sigma^2\epsilon}}{d\,\bfQ_{\mu_0}}(X) - 1 \Big)\\
			&= 1 + \frac{s}{2n} \Big( \Big\{1+\frac{\sigma^2\varepsilon}{\mu_0}\Big\}^{X}
			e^{-\varepsilon} - 1 \Big).
		\end{align}
		Using the inequality $\log(1+t)\le t$, we get
		\begin{equation}
		\KL(\bfQ_\lambda\|\bfQ_{\mu_0}) 
			\le \frac{s}{2n}\int_\NN \Big\{1+\frac{\sigma^2\varepsilon}{\mu_0}\Big\}^{X}
			e^{-\varepsilon}\,\bfQ_\lambda(dX)
			- \frac{s}{2n}.
		\end{equation}
		On the other hand, for every $\mu\in(0,\infty)^p$, we have
		\begin{align}
		\int \Big\{1+\frac{\sigma^2\varepsilon}{\mu_0}\Big\}^{X}\,\bfQ_{\mu}(dX) 
				&= \sum_{k\in\NN} \Big\{1+\frac{\sigma^2\varepsilon}{\mu_0}\Big\}^{k}
				\frac{\mu^{k}}{\sigma^{2k}k!} 
					e^{-\mu/\sigma^2}
				=  \exp\big\{\nicefrac{\mu\varepsilon}{\mu_0}\big\},
		\end{align}
		so that
		\begin{align}
		\KL(\bfQ_\lambda\|\bfQ_{\mu_0}) 
		    &\le \frac{s}{2n}\Big(1-\frac{s}{2n}\Big)
		+\Big(\frac{s^2}{4n^2}\Big)
		 e^{\sigma^2\epsilon^2/\mu_0}-\frac{s}{2n}\\		
		&=\frac{s^2}{4n^2}\big(e^{\sigma^2\varepsilon^2/\mu_0}-1\big).
		\end{align}
		The result follows since, by independence,
		$\KL(\bfP_{\pi_1}\|\bfP_{\pi_0}) = n\KL(\bfQ_\lambda\|\bfQ_{\mu_0})$.
	\end{proof}

\section*{Acknowledgments}
The authors thank an anonymous referee for his/her careful reading of the paper and 
fruitful suggestions. O.\ Collier's research has been conducted as part of the project 
Labex MME-DII (ANR11-LBX-0023-01). 
The work of A. Dalalyan was partially supported by the grant Investissements d'Avenir 
(ANR-11IDEX-0003/Labex Ecodec/ANR-11-LABX-0047).


\bibliography{refs}
\end{document}